\documentclass{birkjour}
 \newtheorem{thm}{Theorem}[section]
 \newtheorem{cor}[thm]{Corollary}
 \newtheorem{lem}[thm]{Lemma}
 
 \newtheorem{defn}[thm]{Definition}
 \newtheorem{rem}[thm]{Remark}

 \numberwithin{equation}{section}

\begin{document}

%-------------------------------------------------------------------------
% editorial commands: to be inserted by the editorial office
%
%\firstpage{1} \volume{228} \Copyrightyear{2004} \DOI{003-0001}
%\seriesextraline{This is the Concrete Title of this Book\br H.E. R and S.T.C. W, Eds.}
%\firstpage{1}

%----------Author 1
\author[Youssef Aserrar  and  Elhoucien Elqorachi]{Youssef Aserrar  and  Elhoucien Elqorachi}

\address{% 
	Ibn Zohr University, Faculty of sciences, 
Department of mathematics,\\
Agadir,
Morocco}

\email{youssefaserrar05@gmail.com, elqorachi@hotmail.com }



\date{January 1, 2020}
\title[ Kannappan-Wilson and Van Vleck-Wilson functional equations]{Kannappan-Wilson and Van Vleck-Wilson functional equations on semigroups}
 
\begin{abstract}
Let $S$ be a semigroup, $Z(S)$ the center of $S$ and $\sigma:S\rightarrow S$ is an involutive automorphism. Our main results is that we describe the solutions of the Kannappan-Wilson functional equation 
\[\displaystyle \int_{S} f(xyt)d\mu(t) +\displaystyle \int_{S} f(\sigma(y)xt)d\mu(t)= 2f(x)g(y),\  x,y\in S,\]
and the Van Vleck-Wilson functional equation
\[\displaystyle \int_{S} f(xyt)d\mu(t) -\displaystyle \int_{S} f(\sigma(y)xt)d\mu(t)= 2f(x)g(y),\  x,y\in S,\]
where  $\mu$ is a measure that is a linear combination of Dirac measures $(\delta_{z_i})_{i\in I}$, such that  $z_i\in Z(S)$ for all $i\in I$. Interesting consequences of these results are presented.

\end{abstract}

%%% ----------------------------------------------------------------------
\maketitle

%%% ----------------------------------------------------------------------
%\tableofcontents
\section{Set up, Notation and terminology}
Throughout this paper $S$ denotes a semigroup, i.e.,  a set equipped with an associative binary operation, $ Z(S)$ denotes  the center of $S$, i.e.,
 $$Z(S):=\left\lbrace s\in S \mid xs=sx\ \text{for all}\ x\in S \right\rbrace .$$
\begin{defn}
Let $f:S\rightarrow\mathbb{C}$.\\
$f$ is  multiplicative, if  $f(xy) = f(x)f(y)$ for all $x, y \in S$.\\
$f$ is additive, if $f(xy) = f(x)+f(y)$ for all $x, y \in S$.\\
$f$ is central, if $f(xy) = f(yx)$ for all $x, y\in S$.\\
$f$ is abelian, if $f$ is central and $f(xyz)=f(xzy)$ for all $x,y,z\in S$.
\end{defn}
 The map $\sigma: S \rightarrow S$ denotes an involutive automorphism of $S$. That is $\sigma(xy)=\sigma(x)\sigma(y)$ and $\sigma(\sigma(x))=x$ for all $x,y\in S$. For any function $f :S \rightarrow \mathbb{C}$ we define the function $f^*:=f\circ \sigma$, and the functions $f^{e}:=\frac{f+f^{*}}{2}$, $f^{\circ}:=\frac{f-f^{*}}{2}$. The function $f$ is said to be even if $f=f^*$ and $f$ is odd if $f=-f^*$. Throughout the paper, we use without explicit mentioning the fact that $\sigma(s)\in Z(S)$ for any $s\in Z(S)$. \par 
If $\chi :S\rightarrow\mathbb{C}$ is a multiplicative function, we  denote by $\phi:S\rightarrow\mathbb{C}$ a non-zero solution of the special sine addition law 
\begin{equation}
\phi (xy)=\phi (x)\chi(y)+\phi(y)\chi(x),\quad x,y\in S.
\label{spsine}
\end{equation}
\section{Introduction}
The classical work of Wilson \cite{W1} was the first contribution to the theory of Wilson's functional equation 
\[f(x+y)+f(x-y)=2f(x)g(y),\quad x,y\in \mathbb{R},\]
which is a generalization of the cosine functional equation
\begin{equation}
g(x+y)+g(x-y)=2g(x)g(y),\quad x,y\in \mathbb{R}.
\label{cos}
\end{equation} 
Aczel's monograph \cite[Section 3.2.1 and Section 3.2.2]{Acz1} discusses the real valued, continuous solutions of Wilson’s equation and contains references to earlier works. These functional equations have been extended to abelian groups, and they have been solved on that setting. An extension of Wilson's functional equation above is treated by Stetk\ae r in \cite[Corollary 5]{ST7} who derived the solution formulas for the functional equation 
\[f(x+y)-f(x+\sigma(y))=2g(x)h(y),\quad x,y\in G,\]
on an abelian group $G$, where $\sigma:G\rightarrow G$ is an involution. The d'Alembert's functional equation 
\[f(xy)+f(x\sigma(y))=2f(x)f(y),\quad x,y\in S,\]
where $\sigma :S\rightarrow S$ is an involutive anti-automorphism (i.e, $\sigma(xy)=\sigma(y)\sigma(x)$ and $\sigma(\sigma(x))=x$ for all $x,y\in S$) has been solved on groups by Davison \cite{Dav} and others. See for example \cite[Chapter 9]{ST1}. It was also studied on semigroups by Stetk\ae r \cite{ST6}.\par 
  In \cite{ST5}, Stetk\ae r solved the variant of d'Alembert's functional equation 
\[f(xy)+f(\sigma(y)x)=2f(x)f(y),\quad x,y\in S,\]
on a semigroup $S$, where $\sigma$ is an involutive automorphism. The solutions are abelian and are of the form $f=\dfrac{\chi+\chi^*}{2}$, where $\chi:S\rightarrow\mathbb{C}$ is a multiplicative function. Elqorachi and Redouani \cite{EA} obtained the solutions of the variant of Wilson's functional equation
\begin{equation}
f(xy)+\tau(y)f(\sigma(y)x)=2f(x)g(y),\quad x,y\in S,
\label{wil}
\end{equation}
on groups, where $\sigma :S\rightarrow S$ is an involutive automorphism, and $\tau :S\rightarrow\mathbb{C}$ is a multiplicative function such that $\tau (x\sigma(x))=1$ for all $x\in S$. It was also solved on groups with $\tau=1$ by Fadli et al. \cite{F1}. The results were extended to semigroups generated by their squares by Ajebbar and Elqorachi \cite{Ajb}. Recently, Aserrar, Chahbi and Elqorachi \cite{Ase1} obtained the solutions on semigroups. Theorem \ref{thm1} reproduces their main points in terms of multiplicative functions, and solutions $\phi:S\rightarrow\mathbb{C}$ of Eq. \eqref{spsine} with $\tau=1$.
\begin{thm}
The solutions $f,g:S\rightarrow\mathbb{C}$ of Eq. \eqref{wil} with $g\neq 0$ and $\tau=1$, namely 
\[f(xy)+f(\sigma(y)x)=2f(x)g(y),\quad x,y\in S,\]
 are the following pairs
\begin{enumerate}
\item[(1)] $f=0$ and $g\neq 0$ arbitrary.
\item[(2)] $f=\alpha \chi+\beta \chi^*$ and $g=\dfrac{\chi+\chi^*}{2}$, where $\alpha,\beta\in\mathbb{C}$ are constants such that $(\alpha ,\beta)\neq (0,0)$ and $\chi:S\rightarrow\mathbb{C}$ is a non-zero multiplicative function.
\item[(3)] $f=\gamma_1\chi+\phi$ and $g=\chi$, where $\gamma_1\in \mathbb{C}$ is a constant, $\chi:S\rightarrow\mathbb{C}$ is a non-zero multiplicative function, $\phi$ is a non-zero solution of the sine addition law \eqref{spsine} such that $\chi^*=\chi$ and $\phi^*=-\phi$.
\label{thm1}
\end{enumerate}
\end{thm}  
Kannappan \cite{Kan1} generalized the cosine functional equation \eqref{cos} to the functional equation
\[f(x+y+z_0)+f(x-y+z_0)=2f(x)f(y),\quad x,y\in \mathbb{R},\]
where $z_0\neq 0$ is a real constant. He showed that any solution $f:\mathbb{R}\rightarrow\mathbb{C}$ of this equation has the form $f(x)=g(x-z_0)$, where $g:\mathbb{R}\rightarrow\mathbb{C}$ is a periodic solution of \eqref{cos} with period $2z_0$.\par 
 In 2017, Stetk\ae r \cite{ST3} generalized the result of Kannappan by solving the functional equation
\[f(xyz_0)+f(xy^*z_0)=2f(x)f(y),\quad x,y\in S,\]
on a semigroup $S$ such that $x\mapsto x^*$ is an involutive anti-automorphism and $z_0\in Z(S)$ fixed. An other similar looking functional equation is Van Vleck's functional equation
\[f(x-y+z_0)-f(x+y+z_0)=2f(x)f(y),\quad x,y\in \mathbb{R},\]
which was studied by Van Vleck \cite{Van} with $z_0>0$ fixed. Stetk\ae r \cite{ST2} extended Van Vleck's result to general semigroups by solving the functional equation 
\[f(xy^*z_0)-f(xyz_0)=2f(x)f(y),\quad x,y\in S,\]
where $z_0\in Z(S)$ is fixed, $x\mapsto x^*$ is an involutive anti-automorphism. He proved that the solutions are all abelian even in the setting of semigroups. Elqorachi \cite{EL} described the solutions of the functional equation 
\[\displaystyle \int_{S} f(xyt)d\mu(t) -\displaystyle \int_{S} f(x\sigma(y)t)d\mu(t)= 2f(x)f(y),\  x,y\in S,\]
where $S$ is a semigroup, $\sigma$ is an involutive automorphism and $\mu$ is a measure that is
a linear combination of Dirac measures $(\delta_{z_i})_{i\in I}$ with $z_i\in Z(S)$ for all $i\in I$. Other results about these functional equations can be found in \cite{BE} and \cite{Pirk}.\par 
 The first purpose of the present paper is to solve the Kannappan-Wilson functional equation 
\begin{equation}
\displaystyle \int_{S} f(xyt)d\mu(t) +\displaystyle \int_{S} f(\sigma(y)xt)d\mu(t)= 2f(x)g(y),\  x,y\in S,
\label{kw}
\end{equation}
where $S$ is a semigroup, $\mu$ is a measure that is
a linear combination of Dirac measures $(\delta_{z_i})_{i\in I}$ with $z_i\in Z(S)$ for all $i\in I$ and $\sigma$ is an involutive automorphism. Equation \eqref{kw} is a generalization of \eqref{wil}, so here we show that the two equations are related and consequently we find explicit formulas for the solutions expressed in terms of multiplicative functions, and solutions of the special sine addition law \eqref{spsine}. As a consequence, we solve the Kannappan-Wilson functional equation
\begin{equation}
f(xyz_0)+f(\sigma(y)xz_0)=2f(x)g(y),\quad x,y\in S,
\label{k}
\end{equation}
on semigroups, where $\sigma$ is an involutive automorphism and $z_0\in Z(S)$ fixed. We also give the solutions of Jensen's functional equation 
\begin{equation}
f(xyz_0)+f(\sigma(y)xz_0)=2f(x),\quad x,y\in S,
\label{jen}
\end{equation}
and by taking $\sigma=id$, we derive the solutions of the functional equation 
\begin{equation}
f(xyz_0)+f(yxz_0)=2f(x)g(y),\quad x,y\in S,
\label{id}
\end{equation}
which is a generalization of the symmetrized multiplicative Cauchy equation
\[f(xy)+f(yx)=2f(x)f(y),\quad x,y\in S,\]
which was  solved by Stetk\ae r in \cite{ST4}.\par 
The second aim of the paper is to solve the Van Velck-Wilson functional equation 
\begin{equation}
\displaystyle \int_{S} f(xyt)d\mu(t) -\displaystyle \int_{S} f(\sigma(y)xt)d\mu(t)= 2f(x)g(y),\  x,y\in S,
\label{Van2}
\end{equation}
and the Van Vleck functional equation 
\begin{equation}
f(xyz_0)-f(\sigma(y)xz_0)=2f(x)g(y),\quad x,y\in S,
\label{Van3}
\end{equation}
where $z_0\in Z(S)$. We also give some results that are not in the literature about the functional equation
\begin{equation}
f(xy)-f(\sigma(y)x)=2f(x)g(y),\quad x,y\in S,
\label{Van1}
\end{equation} 
on a semigroup $S$, where $\sigma:S\rightarrow S$ is an involutive automorphism. We show that the functional equations \eqref{Van2} and \eqref{Van3} are related to \eqref{wil}, and we find explicit formulas for the solutions in terms of multiplicative functions.\par 
Until now, the functional equation \eqref{wil} is the only one among all Wilson's type functional equations which is solved (with $g\neq 0$) on any semigroup, and it is not difficult to see that the treatment of Eq. \eqref{kw} on monoids or groups compared with its treatment on semigroups is much easier thanks to the existence of a neutral element which allows the equation to be related to Eq. \eqref{wil} easily. In the following remark, we give a proof of this fact.
\begin{rem}
Let $f,g:S\rightarrow\mathbb{C}$ be a solution of Eq. \eqref{kw}. Suppose $S$ is a monoid and $e$ its neutral element. If we put $y=e$ in Eq. \eqref{kw} and taking into account that $\sigma(e)=e$, we obtain 
\[\displaystyle \int_{S} f(xt)d\mu(t) +\displaystyle \int_{S} f(xt)d\mu(t)= 2f(x)g(e),\]
for all $x\in S$. That is $\displaystyle \int_{S} f(xt)d\mu(t)= f(x)g(e)$. Thus, Eq. \eqref{kw} becomes 
\[g(e)f(xy)+g(e)f(\sigma(y)x)=2f(x)g(y),\quad x,y\in S.\]
Therefore, if $g(e)=0$ then $f=0$ or $g=0$. Now, if $g(e)\neq 0$, the last equation above can be written as 
\[f(xy)+f(\sigma(y)x)=2f(x)\dfrac{g(y)}{g(e)},\quad x,y\in S,\]
which means that the pair $\left(f,\dfrac{g}{g(e)} \right) $ is a solution of Eq. \eqref{wil}.
\end{rem}
The outline of the paper is as follows. The next section consists of two subsections. In subsection 3.1, we solve the functional equation \eqref{kw}, and as a consequence we find the solutions of \eqref{k}, \eqref{jen} and \eqref{id} in subsection 3.2. Subsection 4.1 contain the solutions of the functional equation \eqref{Van1}, and the solutions of \eqref{Van2} and \eqref{Van3} are given in subsection 4.2.
\section{Kannappan-Wilson functional equations}
\subsection{Solutions of Eq. \eqref{kw}}
The following lemma contain useful results about the solutions of the functional equation \eqref{kw}.
\begin{lem}
Let $f,g:S\rightarrow\mathbb{C}$ be a solution of Eq. \eqref{kw} such that $f\neq 0$ and $g\neq 0$. The following statements hold
\begin{enumerate}
\item[(1)] For all $x,y\in S$
\begin{align}
\begin{split}
\displaystyle \int_{S} f(xt)d\mu(t)\left[ \displaystyle \int_{S} g(ys)d\mu(s)+ \displaystyle \int_{S} g(y\sigma(s))d\mu(s) \right]\\=2f(x)g(y)\displaystyle \int_{S} g(s)d\mu(s).
\end{split}
\label{m1}
\end{align}
\item[(2)] Define $h(x):=\displaystyle \int_{S} f(xt)d\mu(t)$ for all $x\in S$. Then $h\neq 0$.
\item[(3)] Suppose $\displaystyle \int_{S} g(s)d\mu(s)\neq 0$, and define the functions 
$$F(x):=\dfrac{h(x)}{\displaystyle \int_{S} g(s)d\mu(s)}\quad\text{and}\quad G(x):=\dfrac{\displaystyle \int_{S} g(xt)d\mu(t)+ \displaystyle \int_{S} g(x\sigma(t))d\mu(t)}{2\displaystyle \int_{S} g(s)d\mu(s)}.$$
Then $G\neq 0$, and
\begin{equation}
F(xy)+F(\sigma(y)x)=2F(x)G(y),\quad x,y\in S.
\label{i3}
\end{equation}
\item[(4)] If $\displaystyle \int_{S} g(s)d\mu(s)=0$, then $h(x)=\lambda_1 f(x)$ for all $x\in S$, where $\lambda_1 \in \mathbb{C}\backslash \left\lbrace  0\right\rbrace $ is a constant.
\end{enumerate}
\label{L1}
\end{lem}
\begin{proof}
(1) Let $x,y,t,s\in S$ be arbitrary. Applying Eq. \eqref{kw} to the pair $(xyt,s)$ and then to $(\sigma(y)xt,s)$, we get respectively 
\begin{equation}
g(s)f(xyt)= \dfrac{1}{2}\left[\displaystyle \int_{S} f(xytsk)d\mu(k)+\displaystyle \int_{S} f(\sigma(s)xytk)d\mu(k) \right],
\label{i1}
\end{equation} 
and 
\begin{equation}
g(s)f(\sigma(y)xt)=\dfrac{1}{2}\left[\displaystyle \int_{S} f(\sigma(y)xtsk)d\mu(k)+\displaystyle \int_{S} f(\sigma(s)\sigma(y)xtk)d\mu(k) \right].
\label{i2}
\end{equation}
On the other hand, if we apply Eq. \eqref{kw} to the pair $(xt,ys)$ and then to $(xt,y\sigma(s))$, and integrate the result with respect to $t$ and $s$, we obtain respectively 
\begin{align}
\begin{split}
\dfrac{1}{2}\left[\displaystyle \int_{S}\int_{S}\left( \int_{S} f(xytsk)d\mu(k)+\displaystyle \int_{S} f(\sigma(s)\sigma(y)xtk)d\mu(k)\right) d\mu(t)d\mu(s)\right]\\=\displaystyle \int_{S} f(xt)d\mu(t)\displaystyle \int_{S} g(ys)d\mu(s),
\end{split}
\label{m2}
\end{align}
and
\begin{align}
\begin{split}
\dfrac{1}{2}\left[\displaystyle \int_{S}\int_{S}\left( \int_{S} f(\sigma(s)xytk)d\mu(k)+\displaystyle \int_{S} f(\sigma(y)xtsk)d\mu(k)\right) d\mu(t)d\mu(s) \right]\\=\displaystyle \int_{S} f(xt)d\mu(t)\displaystyle \int_{S} g(y\sigma(s))d\mu(s).
\end{split}
\label{m3}
\end{align} 
Thus, by integrating Eq. \eqref{i1} and Eq. \eqref{i2} with respect to $t$ and $s$, then adding the two equations obtained and taking into account  Eq. \eqref{m2} and  Eq. \eqref{m3}, we obtain Eq. \eqref{m1}.\\
(2) Suppose $h=0$. Then Eq. \eqref{kw} yields $2f(x)g(y)=0$ for all $x,y\in S$, so $f=0$ or $g=0$. This is a contradiction since $f\neq 0$ and $g\neq 0$. Thus $h\neq 0$. This occurs in (2).\\
(3) Suppose $\displaystyle \int_{S} g(s)d\mu(s)\neq 0$. Equation  \eqref{i3} follows easily from (1) when we divide Eq. \eqref{m1} by $\left( \displaystyle \int_{S} g(s)d\mu(s)\right) ^2$. In addition, suppose $G=0$. It follows from Eq. \eqref{m1} that 
$$\displaystyle \int_{S} g(s)d\mu(s)\left[\displaystyle \int_{S} f(xyt)d\mu(t)+\displaystyle \int_{S} f(\sigma(y)xt)d\mu(t) \right]=0,$$ 
for all $x,y\in S$. That is $2 f(x)g(y)\displaystyle \int_{S} g(s)d\mu(s)=0$, which implies that $\displaystyle \int_{S} g(s)d\mu(s)=0$ since $f\neq 0$ and $g\neq 0$. This is a contradiction. So $G\neq 0$. This is part (3).\\
(4) Assume that $\displaystyle \int_{S} g(s)d\mu(s)=0$. Replacing $(x,y)$ by $(x,s)$ in the functional equation \eqref{kw} and integrating the result obtained with respect to $s$, we find that 
\[\displaystyle \int_{S}\int_{S} f(xst)d\mu(t)d\mu(s)+\displaystyle \int_{S}\int_{S} f(\sigma(s)xt)d\mu(t)d\mu(s)=0,\]
for all $x\in S$. This implies 
\begin{equation}
\displaystyle \int_{S}\int_{S} f(xst)d\mu(t)d\mu(s)=-\displaystyle \int_{S}\int_{S} f(\sigma(s)xt)d\mu(t)d\mu(s).
\label{m4}
\end{equation}
Now, if we replace $(x,y)$ by $(x,ys)$ in Eq. \eqref{kw} and integrate the result obtained with respect to $s$, we get 
\[\displaystyle \int_{S}\int_{S} f(xyst)d\mu(t)d\mu(s)+\displaystyle \int_{S}\int_{S} f(\sigma(ys)xt)d\mu(t)d\mu(s)=2f(x)\displaystyle \int_{S} g(ys)d\mu(s).\]
That is, in view of Eq. \eqref{m4}
\begin{equation}
\displaystyle \int_{S}\int_{S} f(xyst)d\mu(t)d\mu(s)-\displaystyle \int_{S}\int_{S} f(s\sigma(y)xt)d\mu(t)d\mu(s)=2f(x)\displaystyle \int_{S} g(ys)d\mu(s).
\label{m5}
\end{equation}
On the other hand, if we replace $(x,y)$ by $(xs,y)$ in Eq. \eqref{kw} and integrate the result obtained with respect to $s$, we obtain
\begin{equation}
\displaystyle \int_{S}\int_{S} f(xsyt)d\mu(t)d\mu(s)+\displaystyle \int_{S}\int_{S} f(\sigma(y)xst)d\mu(t)d\mu(s)=2g(y)\displaystyle \int_{S} f(xs)d\mu(s).
\label{m6}
\end{equation} 
Thus, by adding Eq. \eqref{m5} to Eq. \eqref{m6}, we get
\begin{equation}
\displaystyle \int_{S}\int_{S} f(xsyt)d\mu(t)d\mu(s)=f(x)\displaystyle \int_{S} g(ys)d\mu(s)+g(y)\displaystyle \int_{S} f(xs)d\mu(s).
\label{m7}
\end{equation}
Applying Eq. \eqref{m7} to $(xk,y)$ and then to $(x,yk)$ and integrating the result obtained with respect to $k$, we get respectively 
\begin{align*}
\begin{split}
\displaystyle \int_{S}\int_{S}\int_{S} f(xksyt)d\mu(t)d\mu(s)d\mu(k)=\displaystyle \int_{S} f(xk)d\mu(k)\displaystyle \int_{S} g(ys)d\mu(s)\\+g(y)\displaystyle \int_{S}\int_{S} f(xks)d\mu(s)d\mu(k)
\end{split}
\end{align*}
and
\begin{align*}
\begin{split}
\displaystyle \int_{S}\int_{S}\int_{S} f(xsykt)d\mu(t)d\mu(s)d\mu(k)=  f(x)\displaystyle \int_{S}\int_{S} g(yks)d\mu(s)d\mu(k)\\+\displaystyle \int_{S} g(yk) d\mu(k) \displaystyle \int_{S} f(xs)d\mu(s)
\end{split}
\end{align*}
Therefore, comparing the last two identities, we can see that 
\[f(x)\displaystyle \int_{S}\int_{S} g(yks)d\mu(s)d\mu(k)=g(y)\displaystyle \int_{S}\int_{S} f(xks)d\mu(s)d\mu(k).\]
This implies, since $g\neq 0$ that there exists a constant $\alpha \in \mathbb{C}$ such that 
$$\displaystyle \int_{S}\int_{S} f(xks)d\mu(s)d\mu(k)=\alpha f(x)\ \text{ for all}\  x\in S.$$
So, Eq. \eqref{m7} becomes 
\begin{equation}
\alpha f(xy)=f(x)\displaystyle \int_{S} g(ys)d\mu(s)+g(y)\displaystyle \int_{S} f(xs)d\mu(s).
\label{m8}
\end{equation}
If $\alpha =0$, then replacing $x$ by $xk$ in Eq. \eqref{m8} and integrating the result obtained with respect to $k$,  we find that $$\displaystyle \int_{S} f(xk)d\mu(k)\displaystyle \int_{S} g(ys)d\mu(s)=0,$$ 
for all $x,y\in S$. This implies, since $h\neq 0$ that $\displaystyle \int_{S} g(ys)d\mu(s)=0$ for all $y\in S$. Thus, Eq. \eqref{m8} becomes 
$$g(y)\displaystyle \int_{S} f(xs)d\mu(s)=0,$$
 for all $x,y\in S$, which is a contradiction since $h\neq 0$ and $g\neq 0$. Therefore $\alpha \neq 0$, so dividing Eq. \eqref{m8} by $\alpha$ then replacing $y$ by $k$ and integrating the result obtained with respect to $k$, we obtain 
 $$\displaystyle \int_{S} f(xk) d\mu(k)=\dfrac{1}{\alpha}f(x) \displaystyle \int_{S}\int_{S} g(ks)d\mu(s)d\mu(k),$$
  for all $x\in S$. This occurs in case (4) with $\lambda_1:=\dfrac{1}{\alpha}\displaystyle \int_{S}\int_{S} g(ks)d\mu(s)d\mu(k)$. In addition $h\neq 0$ implies $\lambda_1 \neq 0 $. This completes the proof of Lemma \ref{L1}.
\end{proof}
The first main result of the present paper is the following
\begin{thm}
The solutions $f,g:S\rightarrow\mathbb{C}$ of the functional equation \eqref{kw} with $g\neq 0$ can be listed as follows
\begin{enumerate}
\item[(1)] $f=0$ and $g\neq 0$ arbitrary.
\item[(2)] $f=\lambda_2 m$ and $g=\displaystyle \int_{S} m(s) d\mu(s)\dfrac{m+m^*}{2}$, where $m:S\rightarrow\mathbb{C}$ is a multiplicative function and $\lambda_2\in \mathbb{C}\backslash \lbrace 0\rbrace$ a constant such that $\displaystyle \int_{S} m(s) d\mu(s)\neq 0$.
\item[(3)] $f=\alpha_1\chi+\beta_1 \chi^*$ and $g=\displaystyle \int_{S} \chi(s)d\mu(s)\dfrac{\chi+\chi^*}{2}$, where $\chi:S\rightarrow\mathbb{C}$ is a multiplicative function and $\alpha_1,\beta_1\in \mathbb{C}$ are constants such that $\left( \alpha_1,\beta_1\right)\neq (0,0) $,  $\chi\neq \chi^*$,  $\displaystyle \int_{S} \chi(s)d\mu(s)=\displaystyle \int_{S} \chi^*(s)d\mu(s)$ and $\displaystyle \int_{S} \chi(s)d\mu(s)\neq 0$.
\item[(4)] $f=\lambda(\gamma_1\chi+\phi)$ and $g=\displaystyle \int_{S} \chi(s)d\mu(s)\chi$, where $\chi:S\rightarrow\mathbb{C}$ is a non-zero multiplicative function, $\phi$ is a non-zero solution of the sine addition law \eqref{spsine} and $\lambda\in \mathbb{C}\backslash \lbrace 0\rbrace ,\gamma_1\in \mathbb{C}$ are constants such that $\chi^*=\chi$, $\phi^*=-\phi$, $\displaystyle \int_{S} \chi(s)d\mu(s)\neq 0$ and $\displaystyle \int_{S} \phi(s)d\mu(s)=0$.
\end{enumerate}
Note that, of the exceptional case (1), $f$ and $g$ are abelian.
\label{thm}
\end{thm}
\begin{proof}
Let $f,g:S\rightarrow\mathbb{C}$ be a solution of \eqref{kw}. If $f=0$, we can see from Eq. \eqref{kw} that $g\neq 0$ is arbitrary, so we are in family (1). Henceforth we suppose that $f\neq 0$ and $g\neq 0$. We split the discussion into the cases: $\displaystyle \int_{S} g(s)d\mu(s)=0$ and $\displaystyle \int_{S} g(s)d\mu(s)\neq 0$.\\
\underline{Case 1:} $\displaystyle \int_{S} g(s)d\mu(s)=0$. According to Lemma \ref{L1} (4), there exists a constant $\lambda_1\in \mathbb{C}\backslash \lbrace 0\rbrace$ such that $\displaystyle \int_{S} f(xk) d\mu(k)=\lambda_1 f(x)$ for all $x\in S$. So, Eq. \eqref{kw} can be written as 
\[f(xy)+f(\sigma(y)x)=2f(x)\dfrac{g(y)}{\lambda_1},\quad x,y\in S.\]
Then, according to Theorem \ref{thm1} and taking into account that $f\neq 0$, we have two possibilities\\
(i) $f=\alpha \chi+\beta \chi^*$ and $\dfrac{g}{\lambda_1}=\dfrac{\chi+\chi^*}{2}$, where $\chi:S\rightarrow\mathbb{C}$ is a non-zero multiplicative function and $(\alpha ,\beta)\neq (0,0)$ are constants. Since $\displaystyle \int_{S} g(s) d\mu(s)=0$, we can see that $\displaystyle \int_{S} \chi(s) d\mu(s)=-\displaystyle \int_{S} \chi^*(s)d\mu(s)$. Thus, inserting these forms in Eq. \eqref{kw}, we get after some rearrangements that for all $x,y\in S$
\begin{align*}
\begin{split}
\left( \alpha \displaystyle \int_{S} \chi(s) d\mu(s)-\alpha \lambda_1\right) \chi(x)(\chi(y)+\chi^*(y))\\=\chi^*(x)\left( \beta \displaystyle \int_{S} \chi(s) d\mu(s)+\beta\lambda_1\right) (\chi(y)+\chi^*(y)).
\end{split}
\end{align*}
If $\chi=\chi^*$, then $g=\lambda_1 \chi$. So $\displaystyle \int_{S} \chi(s) d\mu(s)=0$ and the identity above reduces to $\alpha =-\beta$. This implies $f=0$, contradiction. So $\chi\neq \chi^*$. Then by the help of \cite[Theorem 3.18]{ST1}, we deduce that 
\begin{equation}
\left\{\begin{array}{l}
\alpha \displaystyle \int_{S} \chi(s) d\mu(s)-\alpha \lambda_1=0 \\
\beta \displaystyle \int_{S} \chi(s) d\mu(s)+\beta\lambda_1=0 
\end{array}\right.
\label{p1}
\end{equation}
Since $f\neq 0$, then at least one of $\alpha$ and $\beta$ is non-zero.\\
(a) Assume $\alpha\neq 0$ and $\beta=0$. Then Eq.\eqref{p1} implies $\lambda_1=\displaystyle \int_{S} \chi(s) d\mu(s)$. This occurs in case (2) with $\lambda_2:=\alpha$ and $m:=\chi$.\\
(b) Suppose $\alpha= 0$ and $\beta\neq 0$. It follows from Eq. \eqref{p1} that $\lambda_1 =-\displaystyle \int_{S} \chi(s) d\mu(s)$. So, we get family (2) with $\lambda_2:=\beta$ and $m:=\chi^*$.\\
(c) If $\alpha\neq 0$ and $\beta\neq 0$. Then Eq. \eqref{p1} implies $\lambda_1=\displaystyle \int_{S} \chi(s) d\mu(s)=-\displaystyle \int_{S} \chi(s) d\mu(s)$. Thus $\lambda_1=0$. This case does not occur.\\
(ii) $f=\gamma_1\chi+\phi$ and $\dfrac{g}{\lambda_1}=\chi$, where $\chi:S\rightarrow\mathbb{C}$ is a non-zero multiplicative function, $\phi$ is a non-zero solution of the sine addition law \eqref{spsine} and $\gamma_1\in \mathbb{C}$ is a constant such that $\chi^*=\chi$ and $\phi^*=-\phi$. In addition $\displaystyle \int_{S} g(s) d\mu(s)=0$ yields $\displaystyle \int_{S} \chi(s) d\mu(s)=0$. Thus, inserting the forms of $f$ and $g$ in Eq. \eqref{kw}, we get after some simplifications
\[\phi=\left( \displaystyle \int_{S} \phi(s) d\mu(s)-\lambda_1 \gamma_1\right) \chi.\]
This implies $\phi=0$, since $\chi^*=\chi$ and $\phi^*=-\phi$,  which is not possible because $\phi\neq 0$. This case does not occur.\\
\underline{Case 2:} $\displaystyle \int_{S} g(s) d\mu(s)\neq 0$. By Lemma \ref{L1} (3), we have
\begin{equation}
F(xy)+F(\sigma(y)x)=2F(x)G(y),\quad x,y\in S,
\label{p2}
\end{equation}
where
$$F(x):=\dfrac{h(x)}{\displaystyle \int_{S} g(s)d\mu(s)}\quad\text{and}\quad G(x):=\dfrac{\displaystyle \int_{S} g(xt)d\mu(t)+ \displaystyle \int_{S} g(x\sigma(t))d\mu(t)}{2\displaystyle \int_{S} g(s)d\mu(s)},$$
and $h(x)=\displaystyle \int_{S} f(xt)d\mu(t)$ for all $x\in S$. In addition, according to Lemma \ref{L1} (2) and (3), $F\neq 0$ and $G\neq 0$. Therefore, applying Theorem \ref{thm1} to Eq. \eqref{p2}, we have two cases to consider\\
(i) $F=\alpha \chi+\beta \chi^*$ and $G=\dfrac{\chi+\chi^*}{2}$, where $\chi:S\rightarrow\mathbb{C}$ is a non-zero multiplicative function and $(\alpha ,\beta)\neq (0,0)$ are constants. On the other hand, since $(f,g)$ is a solution of Eq. \eqref{kw}, we can see that for all $x,y\in S$
\[F(xy)+F(\sigma(y)x)=\dfrac{2f(x)g(y)}{\displaystyle \int_{S} g(s)d\mu(s)}.\]
So, Eq. \eqref{p2} can be written as
\begin{equation}
\dfrac{f(x)g(y)}{\displaystyle \int_{S} g(s)d\mu(s)}=F(x)G(y).
\label{p3}
\end{equation}
Thus, replacing $(x,y)$ by $(x,s)$ in Eq. \eqref{p3} and integrating the result obtained with respect to $s$, we get that
 $$f(x)=F(x)\displaystyle \int_{S} G(s)d\mu(s),$$ 
 for all $x\in S$. That is 
\[f=\dfrac{\alpha \chi+\beta \chi^*}{2}\left(\displaystyle \int_{S} \chi(s)d\mu(s)+\displaystyle \int_{S} \chi^*(s)d\mu(s) \right) . \]
In addition, since $f\neq 0$, we get from Eq. \eqref{p3} that $g=\gamma G$ for some constant $\gamma \in \mathbb{C}\backslash \lbrace 0\rbrace$. That is 
\[g=\dfrac{\gamma}{2}\left( \chi+\chi^*\right). \]
So, $\chi+\chi^*\neq 0$ and $\displaystyle \int_{S} \chi(s)d\mu(s)+\displaystyle \int_{S} \chi^*(s)d\mu(s)\neq 0$, since $g\neq 0$ and $f\neq 0$. Thus, inserting the forms of $f$ and $g$ in Eq. \eqref{kw}, we obtain after some simplifications
\begin{equation}
\alpha \left( \displaystyle \int_{S} \chi(s)d\mu(s)-\gamma\right) \chi+\beta\left( \displaystyle \int_{S} \chi^*(s)d\mu(s)-\gamma\right) \chi^*=0.
\label{p4}
\end{equation}
If $\chi=\chi^*$, then $f=\displaystyle \int_{S} \chi(s)d\mu(s)(\alpha+\beta)\chi$ and the identity above yields $\gamma=\displaystyle \int_{S} \chi(s)d\mu(s)$ since $\alpha+\beta\neq 0$ because $f\neq 0$. This occurs in case (2) with $\lambda_2:=\displaystyle \int_{S} \chi(s)d\mu(s)(\alpha+\beta)$ and $m:=\chi=\chi^*$.\par
Now, if $\chi\neq\chi^*$, then by applying \cite[Theorem 3.18]{ST1}, Eq. \eqref{p4} yields
\begin{equation}
\left\{\begin{array}{l}
\alpha\displaystyle \int_{S} \chi(s)d\mu(s)-\alpha \gamma=0 \\
\beta \displaystyle \int_{S} \chi^*(s)d\mu(s)-\beta\gamma=0 
\end{array}\right.
\label{p5}
\end{equation}
So, we proceed as in case 1 (i) by discussing three cases.\\
(a) $\alpha\neq 0$ and $\beta=0$. We get $\gamma =\displaystyle \int_{S} \chi(s)d\mu(s)$. This is case (2) with $m:=\chi$ and $\lambda_2:=\dfrac{\alpha}{2}\left(\displaystyle \int_{S} \chi(s)d\mu(s)+\displaystyle \int_{S} \chi^*(s)d\mu(s) \right) $.\\
(b) $\alpha= 0$ and $\beta\neq 0$. This implies $\gamma =\displaystyle \int_{S} \chi^*(s)d\mu(s)$. This occurs in case (2) with $m:=\chi^*$ and $\lambda_2:=\dfrac{\beta}{2}\left(\displaystyle \int_{S} \chi(s)d\mu(s)+\displaystyle \int_{S} \chi^*(s)d\mu(s) \right) $.\\
(c) $\alpha\neq 0$ and $\beta\neq 0$. Eq. \eqref{p5} yields $\gamma=\displaystyle \int_{S} \chi(s)d\mu(s)=\displaystyle \int_{S} \chi^*(s)d\mu(s)$. This is case (3) with $\alpha_1:=\alpha\displaystyle \int_{S} \chi(s)d\mu(s)$ and $\beta_1:=\beta\displaystyle \int_{S} \chi(s)d\mu(s)$.\\
 (ii) $F=\gamma_1\chi+\phi$ and $G=\chi$, where $\chi:S\rightarrow\mathbb{C}$ is a non-zero multiplicative function, $\phi$ is a non-zero solution of the sine addition law \eqref{spsine} and $\gamma_1\in \mathbb{C}$ is a constant such that $\chi^*=\chi$ and $\phi^*=-\phi$. Since $f=\displaystyle \int_{S} G(s)d\mu(s)F$ and $g=\gamma G$, where $\gamma \in \mathbb{C}\backslash \lbrace 0\rbrace$ is a constant, we get that 
 \[f=\displaystyle \int_{S} \chi(s)d\mu(s)(\gamma_1\chi+\phi)\quad\text{and}\quad g=\gamma\chi.\]
 Since $f\neq 0$, we can see that $\displaystyle \int_{S} \chi(s)d\mu(s)\neq 0$. Thus, inserting the forms of $f$ and $g$ in Eq. \eqref{kw}, we get after some rearrangements that 
\[\left( \gamma-\displaystyle \int_{S} \chi(s)d\mu(s)\right) \phi=\left( \gamma_1\left( \displaystyle \int_{S} \chi(s)d\mu(s)-\gamma\right) +\displaystyle \int_{S} \phi(s)d\mu(s)\right) \chi.\]
This implies that, since $\chi^*=\chi$ and $\phi^*=-\phi$ 
\[\gamma=\displaystyle \int_{S} \chi(s)d\mu(s)\quad\text{and}\quad \displaystyle \int_{S} \phi(s)d\mu(s)=0.\]
This occurs in case (4) with $\lambda:=\displaystyle \int_{S} \chi(s)d\mu(s)$.\par 
Conversely, if $f$ and $g$ are of the forms (1)--(4) in Theorem \ref{thm} we check by elementary computations that $f$ and $g$ satisfy the functional equation \eqref{kw}. This completes the proof of Theorem \ref{thm}.
\end{proof}
\subsection{Some consequences}

The first consequence of our first main result  reads as follows
\begin{thm}
The solutions $f,g:S\rightarrow\mathbb{C}$ of the functional equation \eqref{k} with $g\neq 0$ can be listed as follows
\begin{enumerate}
\item[(1)] $f=0$ and $g\neq 0$ arbitrary.
\item[(2)] $f=\lambda_2 m$ and $g=\dfrac{m(z_0)}{2}(m+m^*)$, where $m:S\rightarrow\mathbb{C}$ is a multiplicative function and $\lambda_2\in \mathbb{C}\backslash \lbrace 0\rbrace$ a constant such that $m(z_0)\neq 0$.
\item[(3)] $f=\alpha_1\chi+\beta_1 \chi^*$ and $g=\dfrac{\chi(z_0)}{2}(\chi+\chi^*)$, where $\chi:S\rightarrow\mathbb{C}$ is a multiplicative function and $\alpha_1,\beta_1\in \mathbb{C}$ are constants such that $\left(\alpha_1,\beta_1 \right)\neq (0,0) $,  $\chi\neq \chi^*$,  $\chi(z_0)=\chi^*(z_0)$ and $\chi(z_0)\neq 0$.
\item[(4)] $f=\lambda(\gamma_1\chi+\phi)$ and $g=\chi(z_0)\chi$, where $\chi:S\rightarrow\mathbb{C}$ is a non-zero multiplicative function, $\phi$ is a non-zero solution of the sine addition law \eqref{spsine}, $\lambda\in \mathbb{C}\backslash \lbrace 0\rbrace,\gamma_1\in \mathbb{C}$ are constants such that $\chi^*=\chi$, $\phi^*=-\phi$, $\chi(z_0)\neq 0$ and $\phi(z_0)=0$.
\end{enumerate}
Note that, of the exceptional case (1), $f$ and $g$ are abelian.
\label{thm2}
\end{thm}
\begin{proof}
Follows easily from Theorem \ref{thm} by taking $\mu=\delta_{z_0}$, where $z_0$ is a fixed element in $Z(S)$.
\end{proof}
\begin{rem}
If $g=0$, the functional equation \eqref{k} can be written as 
\[f(xyz_0)+f(\sigma(y)xz_0)=0,\quad x,y\in S.\]
Then, for all $x,y,z\in S$
\[f(xyzz_0)=-f(\sigma(yz)xz_0)=f(\sigma(z)xyz_0)=-f(xyzz_0).\] 
Thus, $f(xyzz_0)=0$. So, if $S$ is a regular semigroup, we can see that $f=0$ on $Sz_0:=\lbrace xz_0\  \vert \ x\in S\rbrace$, but $f$ does not vanish necessarily on $S$. See \cite[Remark 3.1]{Ase1}.
\end{rem}
The second consequence of Theorem \ref{thm}  is the following result on the solution of the variant of Jensen's functional equation \eqref{jen}.
\begin{cor}
Let $f:S\rightarrow\mathbb{C}$ be a solution of Eq. \eqref{jen}. Then $f$ has the form $f=\gamma_1+A$, where $A$ is a non-zero additive function and $\gamma_1\in \mathbb{C}$ is a constant such that $A^*=-A$ and $A(z_0)=0$.\par 
Note that, $f$ is abelian.
\label{co1}
\end{cor}
\begin{proof}
Let $f$ be a solution of Eq. \eqref{jen}. By applying Theorem \ref{thm2} with $g=1$ we can see that the only possibility for $g$ to be $1$ is in case (4) of Theorem \ref{thm2}. That is for $\chi=\chi^*=1$. So, $f=\gamma_1+\phi$ where $\gamma_1\in \mathbb{C}$ is a constant and $\phi$ is a non-zero solution of the functional equation 
\[\phi(xy)=\phi(x)+\phi(y),\quad x,y\in S,\]
such that $\phi^*=-\phi$ and $\phi(z_0)=0$. That is $\phi$ additive which occurs in Corollary \ref{co1} with $A:=\phi$.\par 
Conversely, simple computations prove that the formula above for $f$ defines a solution of Eq. \eqref{jen}. This completes the proof. 
\end{proof}
An other interesting consequence is the following corollary about solutions of Eq. \eqref{id}.
\begin{cor}
The solutions $f,g:S\rightarrow\mathbb{C}$ be a solution of Eq. \eqref{id} with $g\neq 0$ are the following
\begin{enumerate}
\item[(1)] $f=0$ and $g\neq 0$ arbitrary.
\item[(2)] $f=\lambda_2 \chi$ and $g=\chi(z_0)\chi$, where $\chi:S\rightarrow\mathbb{C}$ is a non-zero multiplicative function and $\lambda_2\in \mathbb{C}\backslash \lbrace 0\rbrace$ a constant such that $\chi(z_0)\neq 0$.
\end{enumerate}
Note that in case (2), $f$ and $g$ are abelian.
\label{co2}
\end{cor}
\begin{proof}
The result follows directly from Theorem \ref{thm2} by taking $\sigma=id$.
\end{proof}
\section{Van Vleck-Wilson functional equations}
\subsection{Solutions of Eq. \eqref{Van1}} 
The following lemma will be used later
\begin{lem}
Let $F,G:S\rightarrow\mathbb{C}$ be a solution of the functional equation 
\begin{equation}
F(xy)=F(x)G(y)-F(y)G(x)\quad x,y\in S.
\label{u1}
\end{equation}
Then, $F$ and $G$ are linearly dependent.
\label{LV}
\end{lem}
\begin{proof}
Follows from \cite[Theorem 5.2]{Ase2} by taking $\sigma=id$.
\end{proof}
The next theorem gives the general solution of \eqref{Van1}.
\begin{thm}
The solutions $f,g:S\rightarrow\mathbb{C}$ of Eq. \eqref{Van1} with $g\neq 0$, namely 
\[f(xy)-f(\sigma(y)x)=2f(x)g(y),\quad x,y\in S,\]
 are the following
\begin{enumerate}
\item[(1)] $f=0$ and $g\neq 0$ arbitrary.
\item[(2)] $f=\lambda\chi$ and $g=\dfrac{\chi-\chi^*}{2}$, where $\chi:S\rightarrow\mathbb{C}$ is a multiplicative function and $\lambda\in \mathbb{C}\backslash \lbrace 0\rbrace$ is a constant such that $\chi\neq \chi^*$.
\end{enumerate}
Note that in case (2), $f$ and $g$ are abelian.
\label{thmV1}
\end{thm}
\begin{proof}
Let $f,g:S\rightarrow\mathbb{C}$ be a solution of Eq. \eqref{Van1}. If $f=0$, we can see easily from \eqref{Van1} that $g\neq 0$ is arbitrary. This is case (1). From now on, we assume that $f\neq 0$ and $g\neq 0$. Let $x,y\in S$ be arbitrary. By applying Eq. \eqref{Van1} to the pair $(\sigma(y),x)$, we get that 
\[f(\sigma(y)x)-f^*(xy)=2f^*(y)g(x).\]
By adding this to \eqref{Van1} we obtain 
\begin{equation}
f^{\circ}(xy)=f(x)g(y)+f^*(y)g(x).
\label{V1}
\end{equation}
On the other hand, since $f=f^e+f^{\circ}$ and taking \eqref{V1} into account, we get from \eqref{Van1} that 
\begin{align*}
f^e(xy)-f^e(\sigma(y)x)&=2f(x)g(y)-f^{\circ}(xy)+f^{\circ}(\sigma(y)x)\\
&=f(x)g(y)-f^*(y)g(x)+f^{\circ}(\sigma(y)x)\\
&=f(x)g(y)+f^*(x)g^*(y).&
\end{align*}
So, since $f^*=f^e-f^{\circ}$ we deduce that 
\begin{align*}
\begin{split}
f^e(xy)-f^e(\sigma(y)x)=\left[f^e(x)+f^{\circ}(x) \right] \left[g^e(y)+g^{\circ}(y) \right] +\\\left[f^e(x)-f^{\circ}(x) \right] \left[g^e(y)-g^{\circ}(y) \right].
\end{split}
\end{align*}
That is 
\begin{equation}
f^e(xy)-f^e(\sigma(y)x)=2f^e(x)g^e(y)+2f^{\circ}(x)g^{\circ}(y).
\label{V2}
\end{equation}
Now, replacing $(x,y)$ by $(\sigma(y),x)$ in \eqref{V2} and using that $f^e\circ\sigma=f^e$ and $f^{\circ}\circ\sigma=-f^{\circ}$, we obtain 
\begin{equation}
f^e(\sigma(y)x)-f^e(xy)=2f^e(y)g^e(x)-2f^{\circ}(y)g^{\circ}(x).
\label{V3}
\end{equation}
Thus, by comparing Eq. \eqref{V2} with Eq. \eqref{V3}, we see that 
\[f^e(x)g^e(y)+f^{\circ}(x)g^{\circ}(y)=f^{\circ}(y)g^{\circ}(x)-f^e(y)g^e(x).\]
This implies 
\begin{equation}
f^e(x)g^e(y)+f^e(y)g^e(x)=f^{\circ}(y)g^{\circ}(x)-f^{\circ}(x)g^{\circ}(y).
\label{V4}
\end{equation}
In Eq. \eqref{V4}, the left hand side is an even function of $x$ while the right is an odd function of $x$, so
\begin{equation}
f^e(x)g^e(y)+f^e(y)g^e(x)=0,
\label{V5}
\end{equation}
and 
\begin{equation}
f^{\circ}(y)g^{\circ}(x)-f^{\circ}(x)g^{\circ}(y)=0,
\label{V6}
\end{equation}
for all $x,y\in S$. According to \cite[Exercise 1.1 (b)]{ST1} we deduce from \eqref{V5} that we have two possibilities: $f^e=0$ or $g^e=0$.\\
\underline{First case: $f^e=0$}. That is $f=f^{\circ}$, so $f^*=-f$ and from Eq. \eqref{V2} we can see that $g^{\circ}=0$, i.e.,  $g=g^e$. In addition Eq. \eqref{V1} becomes
\[f(xy)=f(x)g(y)-f(y)g(x).\]
According to Lemma \ref{LV}, $f$ and $g$ are linearly dependent. Then $f=0$ or $g=0$, since $f$ is odd and $g$ is even. This case does not occur since $f\neq 0$ and $g\neq 0$.\\
\underline{Second case: $g^e=0$}. That is $g=g^{\circ}$. Since $g\neq 0$, we get from Eq. \eqref{V6} that $f^{\circ}=\lambda g$ for some constant $\lambda \in \mathbb{C}$. Thus, equation \eqref{V1} becomes
\begin{equation}
\lambda g(xy)=f(x)g(y)+f^*(y)g(x).
\label{V7}
\end{equation}
Now, replacing $(x,y)$ by $(x,\sigma(y))$ in \eqref{V7} we obtain, since $g$ is odd 
\begin{equation}
\lambda g(x\sigma(y))=f(y)g(x)-f(x)g(y).
\label{V8}
\end{equation}
If $\lambda=0$ then $f^{\circ}=0$, i.e.,  $f=f^e$. In addition, we get from \eqref{V8} that $f$ and $g$ are linearly dependent. Therefore $f=0$ or $g=0$, since $f$ is even and $g$ is odd. This is a contradiction, since $f\neq 0$ and $g\neq 0$. So $\lambda \neq 0$. Thus, equation \eqref{V8} can be written as
\[g(x\sigma(y))=\dfrac{1}{\lambda} f(y)g(x)-\dfrac{1}{\lambda}f(x)g(y).\]
According to \cite[Theorem 5.2]{Ase2} and taking into account that $g$ is a non-zero odd function, we have two cases to consider\\
Case 1:  $g=\alpha(\chi-\chi^*)$ and $\dfrac{1}{\lambda}f=\dfrac{\chi+\chi^*}{2}+\alpha_1\dfrac{\chi-\chi^*}{2}$, where $\chi:S\rightarrow\mathbb{C}$ is a multiplicative function such that $\chi\neq \chi^*$, $\alpha_1\in \mathbb{C}$ and $\alpha\in \mathbb{C}\backslash \lbrace 0\rbrace$. This implies that $f=\dfrac{\lambda+\lambda \alpha_1}{2}\chi+\dfrac{\lambda-\lambda \alpha_1}{2}\chi^*$. Thus, inserting these forms into \eqref{Van1}, we obtain after some simplifications that 
\[\dfrac{\lambda+\lambda \alpha_1}{2}\chi(x)+\dfrac{\lambda \alpha_1-\lambda}{2}\chi^*(x)=\lambda \alpha(1+\alpha_1)\chi(x)+\lambda \alpha(1-\alpha_1)\chi^*(x),\]
for all $x\in S$. Therefore, since $\chi\neq \chi^*$, by the help of \cite[Theorem 3.18]{ST1} we obtain, since $\lambda \neq 0$ 
\[1+\alpha_1=2 \alpha(1+\alpha_1)\ \ \text{and}\ \ \ \alpha_1-1=2 \alpha(1-\alpha_1).\]
This implies $\alpha_1=2\alpha$. Then $\alpha_1^2=1$. That is $\alpha_1=1$ or $\alpha_1=-1$. So, $\alpha=\dfrac{1}{2}$ or $\alpha=\dfrac{-1}{2}$. Thus 
$$\left\{\begin{array}{l}
f=\lambda\chi \\
g=\dfrac{\chi-\chi^*}{2}
\end{array}\right.\quad\text{or}\quad\left\{\begin{array}{l}
f=\lambda \chi^* \\
g=\dfrac{\chi^*-\chi}{2}
\end{array}.\right.$$
The pair $(f, g)$ on the left falls into case (2). So does the pair on the right, except that there $\chi$ is replaced by $\chi^*$.\\
Case 2: $g=\phi$ and $\dfrac{1}{\lambda}f=\chi+\lambda_2\phi$, where $\chi:S\rightarrow\mathbb{C}$ is a non-zero multiplicative
function, $\phi$ is a non-zero solution of the sine addition law \eqref{spsine} such that $\chi^*=\chi$, $\phi^*=-\phi$ and $\lambda_2\in \mathbb{C}$ is a constant. So $f=\lambda \chi+\lambda c_2\phi$. Therefore, inserting these forms into \eqref{Van1}, we get after some simplifications that 
\[c_2\chi(x)\phi(y)=(\chi(x)+\lambda_2\phi(x))\phi(y),\]
for all $x,y\in S$. This implies, since $\phi\neq 0$ that $(\lambda_2-1)\chi(x)=c_2\phi(x)$ for all $x\in S$. Thus $\lambda_2=1$ and $\lambda_2=0$ since $\chi$ is even and $\phi$ is add, which is impossible. This case does not occur.\par 
Conversely, we check by elementary computations that the pairs described in Theorem \ref{thmV1} are solutions of Eq. \eqref{Van1}. This completes the proof of Theorem \ref{thmV1}.
\end{proof}
\begin{rem}
The only solution $f:S\rightarrow\mathbb{C}$ of the functional equation 
\[f(xy)-f(\sigma(y)x)=2f(x)f(y),\quad x,y\in S,\]
is $f \equiv 0$.
\end{rem}
\subsection{Solutions of Eq. \eqref{Van2} and Eq. \eqref{Van3}} 
In the next lemma we give some key properties of solutions of \eqref{Van2}, namely 
$$\displaystyle \int_{S} f(xyt)d\mu(t) -\displaystyle \int_{S} f(\sigma(y)xt)d\mu(t)= 2f(x)g(y),\  x,y\in S.$$
\begin{lem}
Let $f,g:S\rightarrow\mathbb{C}$ be a solution of Eq. \eqref{Van2} such that $f\neq 0$ and $g\neq 0$. Let $h$ and $F$ be the functions defined in Lemma \ref{L1}. The following statements hold
\begin{enumerate}
\item[(1)] Suppose $\displaystyle \int_{S} g(s)d\mu(s)\neq 0$.Then $F\neq 0$, and
\begin{equation}
F(xy)+F(\sigma(y)x)=2F(x)G_1(y),\quad x,y\in S,
\label{M2}
\end{equation}
where $G_1(x):=\dfrac{\displaystyle \int_{S} g(xt)d\mu(t)- \displaystyle \int_{S} g(x\sigma(t))d\mu(t)}{2\displaystyle \int_{S} g(s)d\mu(s)}$ and $G_1\neq 0$.
\item[(2)] If $\displaystyle \int_{S} g(s)d\mu(s)=0$, then $h(x)=\lambda_1 f(x)$ for all $x\in S$, where $\lambda_1 \in \mathbb{C}\backslash \left\lbrace  0\right\rbrace $ is a constant.
\end{enumerate}
\label{LVan2}
\end{lem}
\begin{proof}
(1) follows by using similar computations to those of the proof of Lemma \ref{L1} (3).\\
(2) By using similar computations to those of the proof of Lemma \ref{L1} (4), we show that 
\begin{equation}
\displaystyle \int_{S}\int_{S} f(xyst)d\mu(t)d\mu(s) -\displaystyle \int_{S}\int_{S} f(\sigma(y)sxt)d\mu(t)d\mu(s)= 2f(x)\displaystyle \int_{S} g(ys)d\mu(s),
\label{D1}
\end{equation}
and 
\begin{equation}
\displaystyle \int_{S}\int_{S} f(xyst)d\mu(t)d\mu(s) -\displaystyle \int_{S}\int_{S} f(\sigma(y)xst)d\mu(t)d\mu(s)= 2g(y)\displaystyle \int_{S} f(xs)d\mu(s).
\label{D2}
\end{equation}
Thus, comparing Eq. \eqref{D1} with Eq. \eqref{D2}, we see that for all $x,y\in S$
$$f(x)\displaystyle \int_{S} g(ys)d\mu(s)=g(y)\displaystyle \int_{S} f(xs)d\mu(s).$$
Now, the result follows easily since $g\neq 0$. This completes the proof.
\end{proof}
Now, we are ready to solve the functional equation \eqref{Van2}.
\begin{thm}
The solutions $f,g:S\rightarrow\mathbb{C}$ of Eq. \eqref{Van2} with $g\neq 0$ can be listed as follows
\begin{enumerate}
\item[(1)] $f=0$ and $g\neq 0$ arbitrary.
\item[(2)] $f=\alpha_1 \chi+\alpha_2\chi^*$ and $g=\displaystyle \int_{S} \chi(s)d\mu(s)\dfrac{\chi-\chi^*}{2}$, where $\chi:S\rightarrow\mathbb{C}$ is a multiplicative function and $\alpha_1,\alpha_2\in \mathbb{C}$ are constants such that $(\alpha_1,\alpha_2)\neq (0,0)$, $\chi\neq \chi^*$ and $\displaystyle \int_{S} \chi(s)d\mu(s)=-\displaystyle \int_{S} \chi^*(s)d\mu(s)\neq 0$.\par 
Note that in case (2), $f$ and $g$ are abelian.
\end{enumerate}
\label{thmVan2}
\end{thm}
\begin{proof}
Let $f,g:S\rightarrow\mathbb{C}$ be a solution of \eqref{Van2}. If $f=0$, then we can see easily from \eqref{Van2} that $g\neq 0$ is arbirary. This occurs in (1). Now, we assume that $f\neq 0$ and $g\neq 0$. We  discuss two cases : 
$$\displaystyle \int_{S} g(s)d\mu(s)=0\ \ \text{and}\ \  \displaystyle \int_{S} g(s)d\mu(s)\neq 0.$$\\
\underline{Case 1:} $\displaystyle \int_{S} g(s)d\mu(s)=0$. According to Lemma \ref{LVan2} (2), there exists a constant $\lambda_1\in \mathbb{C}\backslash \lbrace 0\rbrace$ such that $\displaystyle \int_{S} f(xk) d\mu(k)=\lambda_1 f(x)$ for all $x\in S$. Thus, Eq. \eqref{Van2} can be written as 
\[f(xy)-f(\sigma(y)x)=2f(x)\dfrac{g(y)}{\lambda_1},\quad x,y\in S.\]
Then, according to Theorem \ref{thmV1} and taking into account that $f\neq 0$, we have one possibility
\[f=\lambda \chi\ \ \text{and}\ \ \dfrac{g}{\lambda_1}=\dfrac{\chi-\chi^*}{2},\]
where $\chi:S\rightarrow\mathbb{C}$ is a multiplicative function and $\lambda\in \mathbb{C}\backslash \lbrace 0\rbrace$ is a constant such that $\chi\neq \chi^*$. Inserting the forms of $f$ and $g$ into Eq. \eqref{Van2}, we obtain after some simplifications that $\lambda_1 =\displaystyle \int_{S} \chi(s)d\mu(s)$. This is part (2) with $\alpha:=\lambda$.\\
\underline{Case 2:} $\displaystyle \int_{S} g(s)d\mu(s)\neq 0$. By Lemma \ref{LVan2} (1), we have that 
$$F(xy)+F(\sigma(y)x)=2F(x)G_1(y),\quad x,y\in S,$$
where
$$F(x):=\dfrac{h(x)}{\displaystyle \int_{S} g(s)d\mu(s)}\quad\text{and}\quad G_1(x):=\dfrac{\displaystyle \int_{S} g(xt)d\mu(t)- \displaystyle \int_{S} g(x\sigma(t))d\mu(t)}{2\displaystyle \int_{S} g(s)d\mu(s)},$$
and $h(x):=\displaystyle \int_{S} f(xt)d\mu(t)$ for all $x\in S$. Thus, according to Theorem \ref{thm1} and taking into account that $F\neq 0$ and $G_1\neq 0$, we get 
$$F=\alpha \chi+\beta \chi^*\quad \text{and}\quad G_1=\dfrac{\chi+\chi^*}{2},$$
 where $\alpha,\beta\in\mathbb{C}$ are constants such that $(\alpha ,\beta)\neq (0,0)$ and $\chi:S\rightarrow\mathbb{C}$ is a non-zero multiplicative function. Therefore, proceeding similarly to the proof of Theorem \ref{thm} case (2)(i), we find that 
\[f=\alpha_1\chi+\alpha_2\chi^*\ \ \text{and}\ \ g=\displaystyle \int_{S} \chi(s)d\mu(s)\dfrac{\chi-\chi^*}{2},\]
where  $\alpha_1,\alpha_2\in \mathbb{C}$ are constants such that $(\alpha_1,\alpha_2)\neq (0,0)$, $\chi\neq \chi^*$ and $\displaystyle \int_{S} \chi(s)d\mu(s)=-\displaystyle \int_{S} \chi^*(s)d\mu(s)\neq 0$. This occurs in case (2).\\
The second possibility according to Theorem \ref{thm1} is that $F=\gamma(c\chi+\phi)$ and $G_1=\chi$, where $\gamma\in \mathbb{C}\backslash \lbrace 0\rbrace$ is a constant. This leads to the contradiction $g=0$.\par
The converses are straightforward verifications. This completes the proof.
\end{proof}
The final result of the paper is the following corollary about the solutions of the functional equation \eqref{Van3}, namely 
$$f(xyz_0)-f(\sigma(y)xz_0)=2f(x)g(y),\quad x,y\in S.$$
\begin{cor}
The solutions $f,g:S\rightarrow\mathbb{C}$ of Eq. \eqref{Van3} with $g\neq 0$ can be listed as follows
\begin{enumerate}
\item[(1)] $f=0$ and $g\neq 0$ arbitrary.
\item[(2)] $f=\alpha_1 \chi+\alpha_2\chi^*$ and $g=\chi(z_0)\dfrac{\chi-\chi^*}{2}$, where $\chi:S\rightarrow\mathbb{C}$ is a multiplicative function and $\alpha_1,\alpha_2\in \mathbb{C}$ are constants such that $(\alpha_1,\alpha_2)\neq (0,0)$, $\chi\neq \chi^*$ and $\chi(z_0)=-\chi^*(z_0)\neq 0$.\par 
Note that in case (2), $f$ and $g$ are abelian.
\end{enumerate}
\end{cor}
\begin{proof}
The result follows from Theorem \ref{thmVan2} by taking $\mu=\delta_{z_0}$.
\end{proof}
\subsection*{Acknowledgement} Our sincere regards and gratitude go to Professor Henrik Stetk\ae r for many valuable comments on our papers. We would also like to express our thanks to the
referees for useful comments.

% ------------------------------------------------------------------------
\end{document}